\documentclass{amsart}

\usepackage{amssymb,amsbsy,amsthm,amsmath,graphicx,epsfig,color}
\usepackage{verbatim}
\usepackage[T1]{fontenc}

\newtheorem{thm}{Theorem}[section]
\newtheorem*{thm*}{Theorem}

\newtheorem{lemma}[thm]{Lemma}

\theoremstyle{remark}

\newcommand{\id}{\mathrm{id}}

\newcommand{\N}{\mathbb{N}}

\newcommand{\veps}{\varepsilon}

\newcommand{\la}{\langle}
\newcommand{\ra}{\rangle}

\newcommand{\Eins}{\mathbf{1}}

\newcommand{\TT}{\mathcal{T}}
\newcommand{\TTinv}{\mathcal{T}_{\text{inv}}}
\newcommand{\BS}{\mathcal{BS}}

\newcommand{\rg}{\text{rg}}

\newcommand{\sot}{\texttt{SOT}}
\newcommand{\wot}{\texttt{WOT}}

\title{A generic transformation is invertible}

\author{Tanja Eisner}
\address{Institute of Mathematics, University of Leipzig,
P.O. Box 100 920, 04009 Leipzig, Germany}
\email{eisner@math.uni-leipzig.de}


\begin{document}

%
%

\maketitle

\begin{abstract} We show that, on a standard non-atomic probability space,  inver\-tible measure-preserving transformations form a dense $G_\delta$ 
subset of the space of all measure-preserving transformations endowed with the strong ($=$weak) operator topology. This implies that all properties which are generic for inver\-tible transformations are also generic for general ones. We further show that invertible Koopman operators form a dense $G_\delta$ subset of all bi-stochastic ope\-rators for the weak operator topology, and the same holds for general Koopman operators.
\end{abstract}

\section{Introduction}

Let $(X,\mu)$ be a standard non-atomic probability space and $T:X\to X$ be measure preserving, i.e., satisfies $\mu(T^{-1}(A))=\mu(A)$ for every measurable $A\subset X$. Recall that the corresponding \emph{Koopman operator}\footnote{We denote it by the same letter as the transformation following a tradition in ergodic theory.} $T$ on $L^2(X,\mu)$ is defined by $(Tf)(x):=f(Tx)$ and is an isometry on $L^2(X,\mu)$. Moreover, this operator  is invertible, i.e., unitary, if and only if\footnote{see, e.g., \cite[Rem.~8.48(e)]{EF-book} for the ``only if'' direction} the underlying transformation is essentially invertible. 

Denote by $\TT$ the set of all measure-preserving transformations on $(X,\mu)$. Then  $\TT$ is a complete metric\footnote{An example of a metric inducing this topology is 
$$
d(T,S):=\sum_{j=1}^\infty \frac{\|Tf_j-Sf_j\|}{2^{j}\|f_j\|}
$$
for a fixed dense subset $(f_j)$ of $L^2(X,\mu)\setminus \{0\}$.} space with respect to the strong ($=$weak) operator topology for the corresponding Koopman operators, cf.~Halmos \cite{Halmos-book}. 

Usually one works with the space $\TTinv$ of all invertible measure-preserving transformations on $(X,\mu)$ with respect to the strong$^*$ ($=$weak) operator topology\footnote{One says that a net $(T_\alpha)$ of operators on a Hilbert space converges \emph{strongly$^*$} to an operator $T$ if $\lim_\alpha T_\alpha=T$ and $\lim_\alpha T_\alpha^*=T^*$ strongly.} for the corresponding Koopman operators which is also a complete metric space, see Halmos \cite{Halmos-book}, and studies generic properties of transformations therein.  
Since the existence proof by Halmos \cite{Halmos44} and Rokhlin \cite{Rohlin48} of a weakly mixing but not mixing invertible transformation (with a first concrete example constructed much later by Chacon \cite{Chacon}), the study of generic properties in  $\TTinv$ has been an active and exci\-ting area of research. For example, the following further properties are known to be generic in $\TTinv$: zero entropy (Rohlin \cite{Rohlin59}), rigidity (Katok \cite{Katok85}, Nadkarni \cite[Chapter 8]{Nadkarni-book}), 
having roots of all orders (see King \cite{King00}) and embedding into many flows (see de la Rue, de Sam Lazaro \cite{delaRue-deSamLazaro03}, Stepin, Eremenko \cite{StepinEremenko04}). For more results see, e.g., Stepin \cite{Stepin86}, Choksi, Nadkarni \cite{ChoksiNadkarni90}, Nadkarni \cite{Nadkarni-book}, del Junco, Lema\'nczyk \cite{delJuncoLemanczyk92}, Ageev \cite{Ageev03},  Ryzhikov \cite{Ryzhikov97,Ryzhikov24}, Solecki \cite{Solecki14,Solecki23}. For generic properties of extensions we refer to Schnurr \cite{Schnurr19}, Glasner, Weiss \cite{GlasnerWeiss19}, Ryzhikov \cite{Ryzhikov23,Ryzhikov23b}, Kozyreva, Ryzhikov \cite{KozyrevaRyzhikov25}.

Non-invertible transformations are harder to work with than invertible ones and are not so well-studied. We will show that nevertheless all generic properties from $\TTinv$ transfer to $\TT$ by the following.

\begin{thm}[Invertibility is generic]\label{thm:invgen}
The set $\TTinv$ of all invertible measure-preserving transformations is a dense $G_\delta$ subset of $\TT$ with respect to the strong ($=$weak) operator topology for the corresponding Koopman operators.
\end{thm}
In particular, the set $\TTinv$ is residual in $\TT$. 
%
We present two proofs of the $G_\delta$-property, a short one 
 based on the fact that an isometry is invertible if and only if it has dense range, and a longer one using the points of continuity method.

An operator theoretic analogue of Theorem \ref{thm:invgen} with a residuality statement instead of the dense $G_\delta$ property was shown in \cite{E10}, with the corresponding $G_\delta$-refinement  essentially done by Grivaux, Matheron, Menet \cite{GrivauxMatheronMenet22}, see \cite[Theorem 7.1]{EisnerGillet}. For a related result for continuous maps on the unit interval preserving the Lebesgue measure see Bobok, \v Cin\v c, Oprocha, Troubetzkoy \cite{BobokCincOprochaTroubetzkoy}.
 
A linear operator $P$ on $L^2(X,\mu)$ is called \emph{bi-stochastic} if $P\Eins=P^*\Eins=\Eins$ and $P$ is positive, i.e., satisfies $Pf\geq 0$ for every $f\geq 0$. Recall that such operators are contractive on $L^\infty(X,\mu)$, $L^1(X,\mu)$ and hence by interpolation on every $L^p(X,\mu)$, $p\in[1,\infty]$. Clearly, every Koopman operator is bi-stochastic, and the set $\BS$ of all bi-stochastic operators on $L^2(X,\mu)$ is a complete metric space with respect to the weak operator topology. The following shows the genericity of Koopman operators in this space, where we notationally identify a measure-preserving transformation with the corresponding Koopman operator.

\begin{thm}[Being Koopman is generic]\label{thm:Koopman-gen}
Both 
$\TTinv$ and 
$\TT$ are  
dense $G_\delta$ subsets of $\BS$ with respect to the weak operator topology.
\end{thm}

Also here we present two proofs of the $G_\delta$ property.
For the operator theore\-tic analogue of Theorem \ref{thm:Koopman-gen} see \cite[Theorem 7.1]{EisnerGillet} (with the residuality statement done in \cite{E10}). 


\textbf{Acknowledgements.} The author is very grateful to Ethan Akin for a crystallisation of the argument of the first proof in Section \ref{sec:Gdelta-1} from its original version concerning residuality. 
She is indebted to Valentin Gillet for very helpful discu\-ssions regarding the points of continuity method in the operator theoretic context and, in particular, for the references  \cite{GrivauxMatheronMenet21b,GrivauxMatheronMenet22} and the proof of Lemma \ref{lem:T-sep-SOT*}. She also thanks Raj Dahya for his question regarding genericity of  Koopman operators in all bi-stochastic operators resulting in Theorem \ref{thm:Koopman-gen} as well as the referee for careful reading and helpful comments.

\section{Proof: Density}\label{sec:proof-density}

The following result describes the weak closure of $\TTinv$, see Vershik \cite[Sections 1--2]{Vershik77} and a simplification in Ryzhikov \cite[Remark after Lemma 5.5]{Ryzhikov97b}. We present the details for the reader's convenience. 
\begin{lemma}\label{lem:weak-closure-discrete}
The Koopman operators of invertible measure-preserving transformations are dense in the set of all bi-stochastic operators on $L^2(X,\mu)$ with respect to the weak operator topology.  
\end{lemma}
\begin{proof}
Assume without loss of generality that $X=[0,1]$ and $\mu$ is the Lebesgue measure, see, e.g., \cite[Rem.~8.48(d)]{EF-book}. Let $P$ be a bi-stochastic operator on $L^2[0,1]$, $n\in\N$ and divide $[0,1]$ into $2^n$ dyadic intervals $I_j:=\left[\frac{j-1}{2^n},\frac{j}{2^n}\right]$, $j\in\{1,\ldots,2^n\}$. We call the characteristic functions of such intervals \emph{dyadic functions of degree $n$}. Denote $
a_{jk}:=
\la P \Eins_{I_j},\Eins_{I_k}\ra$ for $j,k\in\{1,\ldots,2^n\}$.

 Since $P$ is bi-stochastic, the matrix $(2^na_{jk})$ is bi-stochastic, i.e., $a_{jk}\geq 0$ and $\sum_{j}a_{jk}=\sum_{k}a_{jk}=\frac1{2^n}$ holds for all $j,k$.  It is now easy to construct a measure-preserving transformation $T$ on $[0,1]$ satisfying
\begin{equation}\label{eq:T=P}
\la T\Eins_{I_j},\Eins_{I_k}\ra =
\la P\Eins_{I_j},\Eins_{I_k}\ra\quad \forall j,k\in\{1,\ldots,2^n\}
\end{equation}
as follows.
Divide the interval $I_1$ into subintervals $I_{j,1}$ of length $a_{j1}$, $j\in\{1,\ldots, 2^n\}$, which we count from left to right. Define $T$ so that it leaves the first subinterval $I_{11}$ as it is, shifts the second subinterval $I_{2,1}$ to the beginning of $I_2$, etc. 
Analogously define $T$ on $I_2$ shifting the  subinterval $I_{1,2}$ of length $a_{12}$ to the beginning of the not yet covered part of $I_1$ (intersecting the covered part in one end point), shifting the subinterval of length $a_{22}$ to the left of the not yet covered part of $I_2$ etc. 
In this way, for every $n$ we obtain an invertible measure-preserving transformation $T_n$ on $[0,1]$ satisfying \eqref{eq:T=P} by
$$
\la T\Eins_{I_j},\Eins_{I_k}\ra=\mu(T^{-1}(I_j)\cap I_k)=\mu(I_{j,k})=a_{jk}\quad \forall j,k.
$$
Note that by \eqref{eq:T=P} and linearity  
$\la T f,g\ra =\la Pf,g\ra$ holds 
for all $f,g$ being linear combinations of dyadic functions of degree 
$\leq n$.
Since the linear span of dyadic functions is dense in $L^2[0,1]$, the sequence $(T_n)$ approximates $P$ weakly.
\end{proof}
 

\section{Theorem \ref{thm:invgen}: 
$G_\delta$ property}\label{sec:Gdelta-1}

\subsection{First proof}

The first proof of the $G_\delta$ property of $\TTinv$ in $\TT$ is inspired by 
\cite{E10}. 

Let $(f_j)$ be dense in $L^2(X,\mu)$. Since every isometry has closed range and $(Tf_j)$ is dense in it, we have
\begin{eqnarray*}
\TTinv
=\bigcap_{j,k\in\N}\bigcup_{l\in\N}\left\{T\in \TT:\, \|f_j-Tf_l\|<\frac1k\right\}
=:\bigcap_{j,k\in\N}\bigcup_{l\in\N}M_{j,k,l}.
\end{eqnarray*}
It is clear that every set $M_{j,k,l}$ is open for the strong operator topology, finishing the proof.



\subsection{Second proof}

We 
also provide an alternative proof of the $G_\delta$ property via the points of continuity method.

The following is an ergodic theoretic analogue of Grivaux, Matheron, Menet \cite[Prop.~2.11]{GrivauxMatheronMenet22}.
Here and later we denote by $\sot$ and $\sot^*$  the strong and the strong$^*$ operator topology on the Koopman operators, respectively.
\begin{lemma}\label{lem:points-cont}
The set $\TTinv$ coincides with the set of continuity points of the identity map $\text{id}:(\TT,\sot)\to (\TT,\sot^*)$. 
\end{lemma}
\begin{proof}
Let $T\in \TTinv$ and $(T_n)\subset \TT$ converging to $T$ in the strong operator topo\-logy. Then it converges also in the weak operator topology implying $T_n^*\to T^*$ weakly. Since $T$ is by assumption invertible, $T^*$ is an isometry as well and hence\footnote{Recall that weak (operator) convergence of contractions on a Hilbert space implies strong convergence as soon as the limit operator is an isometry.} $T_n^*\to T^*$ strongly.  Thus $T$ is a point of continuity of $\text{id}:(\TT,\sot)\to (\TT,\sot^*)$.  

Let now $T\in\TT$ be such a point of continuity  and let $f\in L^2(X,\mu)$. 
 By Lemma \ref{lem:weak-closure-discrete} there exists a sequence $(T_n)\subset \TTinv$ converging to $T$ weakly and hence strongly. By the point of continuity assumption we have
$$
\|f\|=\lim_{n\to\infty}\|T_n^*f\|=\|T^*f\|,
$$
and in particular $T^*$ is injective. Equivalently, $T$ has dense range, so\footnote{Recall that isometries always have closed range.} is surjective and hence invertible. This shows $T\in\TTinv$.
\end{proof}

The idea of the following proof in the operator theoretic setting, being a simplification of the one by Grivaux, Matheron, Menet \cite[beginning of Sect.~3]{GrivauxMatheronMenet21b}, was communicated to the author by Valentin Gillet.

\begin{lemma}\label{lem:T-sep-SOT*}
The space $(\TT,\sot^*)$ is separable. 
\end{lemma}
\begin{proof}
Let $(f_k)_{k=1}^\infty$ be dense in $L^2(X,\mu)$ and consider the set 
$$
A:=\{((Tf_k)_{k=1}^\infty,(T^*f_k)_{k=1}^\infty):\, T\in\TT\}\subset  (L^2(X,\mu))^\N\times (L^2(X,\mu))^\N=:Y.
$$
Observe first that $Y$ is a Polish space, where $(L^2(X,\mu))^\N$ is endowed with the product topology induced by the norm topology on $L^2(X,\mu)$. Since for metric spaces separability is equivalent to second countability, $A$ is separable as well. 

Thus let $(a_n)$ be dense in $A$ and denote by $(T_n)$ the corresponding countable subset 
of $\TT$. We show that $(T_n)$ is dense in $(\TT,\sot^*)$. 
Let $T\in\TT$ and let $(n_l)$ be a subsequence of $\N$ such that $(a_{n_l})$ approximates $((Tf_k)_{k=1}^\infty,(T^*f_k)_{k=1}^\infty)$ in $Y$. This means that $(T_{n_l})$ approximates $T$ and $(T^*_{n_l})$ approximates $T^*$ in $\sot$, i.e., $(T_{n_l})$ approximates $T$ in  $\sot^*$. The proof is complete.
\end{proof}

\begin{proof}[Proof of Theorem \ref{thm:invgen}]

By a classical result of Baire, see Kechris \cite[Thm.~24.14]{Kechris-book}, the set of continuity points of a Baire class $1$ map\footnote{A map $f:X\to Y$ between topological spaces $X$ and $Y$ is called \emph{of Baire class $1$} (or \emph{Borel $1$}) if the preimage of every open subset of $Y$ is $F_\sigma$ in $X$.} $f:X\to Y$ form a dense $G_\delta$ subset of $X$ whenever $X,Y$ are metrizable and $Y$ is separable. By Lemmata \ref{lem:points-cont} and \ref{lem:T-sep-SOT*} it just remains to check the Baire class $1$ property of the identity map $\id:(\TT,\sot)\to(\TT,\sot^*)$. This can be done as in \cite[Lemma 2.9]{GrivauxMatheronMenet22} and we repeat here the simple argument for completeness. 

Every $\sot^*$-open subset of $\TT$ can be written as a countable union of finite intersections of sets either of the form 
\begin{equation*}
U_{S,f,\veps}:=\left\{T\in \TT:\, \|(T-S)f\|<\veps\right\}=\bigcup_{n=1}^\infty\left\{T\in \TT:\, \|(T-S)f\|\leq \veps-\frac1n\right\}
\end{equation*}
or of the form 
\begin{eqnarray*}
U^*_{S,f,\veps}:=\{T\in \TT:\, \|(T^*-S^*)f\|<\veps\}=\bigcup_{n=1}^\infty\left\{T\in \TT:\, \|(T^*-S^*)f\|\leq \veps-\frac1n\right\}
\end{eqnarray*}
for $S\in\TT$, $f\in L^2(X,\mu)$ and $\veps>0$. Since $\|(T^*-S^*)f\|\leq \veps-\frac1n$ can be written as 
$$
|\la (T-S)g,f\ra|=|\la g,(T^*-S^*)f\ra|\leq \veps-\frac1n \qquad \forall g\in L^2(X,\mu) \text{ with }\|g\|=1,
$$
the $\sot$-$F_\sigma$ property of $\sot^*$-open sets is clear. Thus $\id:(\TT,\sot)\to(\TT,\sot^*)$ is of Baire class $1$. 
%
\end{proof}

\section{Proof of Theorem \ref{thm:Koopman-gen}}
\label{sec:bi-stoch}

We already saw in Lemma \ref{lem:weak-closure-discrete} that $\TTinv$ (and hence also $\TT$) is dense in $\BS$ with respect to the weak operator topology, where we again notationally identify a measure-preserving transformation with its Koopman operator. For the $G_\delta$ pro\-perty we will again provide two proofs.

\subsection{First proof}

We first show the $G_\delta$ property of  $\TT$. 

It is well known that an operator $T\in \BS$ is a (not necessarily invertible) Koopman operator 
if and only if it is an isometry. Indeed, by Ionescu Tulcea\footnote{Ionescu Tulcea considered invertible isometries/nonsingular transformations, but the same proof (using Lamperti \cite[Proof of Thm.~3.1]{Lamperti58}, cf.~Royden \cite[Proof of Thm.~15.24]{Royden-book}) works also in the general case.} \cite[Footnote 3]{IonescuTulcea64}, every positive isometry on $L^2(X,\mu)$ has the form 
$$(Tf)(x)=h(x)f(\tau x)$$ 
for some $h\in L^2(X,\mu)$ with $h\geq 0$ and some 
null preserving\footnote{A measurable transformation is called \emph{null preserving} if its preimages of null sets are null sets.} transformation $\tau:X\to X$. Now, $T\Eins=\Eins$ implies $h=\Eins$ and $T'\Eins=\Eins$ implies $\int_X f\circ \tau \,d\mu=\int_X f\,d\mu$, i.e., $\tau$ is measure preserving and hence $T$ is Koopman.  

Thus 
$$
\TT=\bigcap_{j,k\in\N}\left\{T\in \BS:\ \|Tf_j\|> \frac{k}{k+1}\|f_j\|\right\}=:\bigcap_{j,k\in\N} M_{j,k}
$$
for a fixed dense sequence $(f_j)$ in $L^2(X,\mu)\setminus\{0\}$. We show that $\TT$ belongs to the interior of each $M_{j,k}$, i.e., that
$$
\TT=\bigcap_{j,k\in\N} \text{int} M_{j,k}
$$
holds (which clearly implies the $G_\delta$ property). Assume that for some $j,k$ this is not the case and there exists a sequence $(T_n)\subset \BS\setminus M_{j,k}$ converging to $T\in \TT$ weakly. Then it also converges to $T$ strongly and hence $\|Tf_j\|\leq \frac{k}{k+1}\|f_j\|$, a contradiction ($T$ is an isometry).  

To show the $G_\delta$ property of $\TTinv$, observe that a Koopman operator is invertible if and only if its adjoint is again a Koopman operator.  Indeed, the ``only if'' direction is clear and the ``if'' direction follows from $\rg T=\overline{\rg T}=(\ker T^*)^\perp$ for an isometry $T$. Thus we have 
$$
\TTinv=\TT \cap \TT^*.
$$
Since the map $T\mapsto T^*$ is an isomorphism\footnote{It is easy to see that the adjoint of a positive operator is positive using the fact that $f\geq 0$ holds if and only if $\la f,g\ra\geq 0$ for every $g\geq 0$.} on $\BS$ with respect to the weak operator topology, the $G_\delta$ property of $\TTinv$ follows from the $G_\delta$ property of $\TT$ shown above.

\subsection{Second proof}
This proof is again based on 
the points of continuity method. 
Hereby we denote by $\wot$ the weak operator topology on $\BS$.

\begin{lemma}
The sets $\TTinv$ and  $\TT$ coincide with the set of continuity points of the map $\id:(\BS,\wot)\to (\BS, \sot^*)$ and the map $\id:(\BS,\wot)\to (\BS, \sot)$, respectively. 
\end{lemma}
\begin{proof}
We will show the assertion for $\TTinv$ while the proof for $\TT$ is analogous. Let $(T_n)\subset \BS$ converge to $T\in \TTinv$ weakly, implying also that $(T_n^*)$ converges weakly to $T^*$. Since both $T$ and $T^*$ are isometries, we see that both sequences converge in fact strongly, i.e., $T$ is a point of continuity of $\id:(\BS,\wot)\to (\BS, \sot^*)$. 

Conversely, let $T\in \BS$ be a point of continuity of $\id:(\BS,\wot)\to (\BS, \sot^*)$. By Lemma \ref{lem:weak-closure-discrete} there exists a sequence $(T_n)\subset \TTinv$ converging to $T$ weakly. By the point of continuity assumption we see that $(T_n)$ converges to $T$ strongly$^*$. Since $\TTinv$ is closed with respect to the strong$^*$ topology, $T\in \TTinv$.
\end{proof}

We now argue as in the proof of Theorem \ref{thm:invgen}. The proof of the Baire class $1$ property for the maps $\id: (\BS,\wot)\to(\BS, \sot^*)$ and $\id: (\BS,\wot)\to(\BS, \sot)$ follows as in the proof of Theorem \ref{thm:invgen} by representing the sets 
$U_{S,f,\veps}:=\{T\in\BS:\, \|(T-S)f\|<\veps\}$ for $S\in\BS$, $f\in L^2(X,\mu)$ and $\veps>0$ as 
$$
\bigcup_{n=1}^\infty \left\{T\in \BS:\, |\la(T-S)f,g\ra|\leq \veps-\frac1n \quad \forall g\text{ with $\|g\|\leq 1$}\right\}
$$
and the sets $U^*_{S,f,\veps}:=\{T\in\BS:\, \|(T^*-S^*)f\|<\veps\}$ as 
$$
\bigcup_{n=1}^\infty \left\{T\in \BS:\, |\la(T-S)g,f\ra|=|\la g,(T^*-S^*)f\ra|\leq \veps-\frac1n \quad \forall g\text{ with $\|g\|\leq 1$}\right\}.
$$
The same argument as in the proof of Lemma \ref{lem:T-sep-SOT*} (replacing $\TT$ by $\BS$) shows that $\BS$ is separable with respect to $\sot^*$ (and hence also $\sot$). Thus the $G_\delta$ property of both $\TTinv$ and $\TT$ follows from Baire's result, see Kechris \cite[Thm.~24.14]{Kechris-book}.

\end{document}